\documentclass[12pt, final]{l4dc2020} 

\makeatletter
\def\set@curr@file#1{\def\@curr@file{#1}} 
\makeatother
\usepackage[load-configurations=version-1]{siunitx} 

\title[Hamilton--Jacobi--Bellman Equations for Q-Learning]{Hamilton-Jacobi-Bellman Equations for Q-Learning\\ in Continuous Time}
\usepackage{times}
\usepackage{mathrsfs,enumerate}
\usepackage{amsfonts,amssymb,amscd,bbm}
\usepackage{float}
\usepackage{caption}
\usepackage{bm}
\usepackage{algorithm2e}

\author{%
 \Name{Jeongho Kim} \Email{jhkim206@snu.ac.kr}\\
 \addr Institute of New Media and Communications, Seoul National University, Seoul  08826, South Korea
 \AND
 \Name{Insoon Yang} \Email{insoonyang@snu.ac.kr}\\
 \addr Department of Electrical and Computer Engineering, Automation and System Research Institute, Seoul National University, Seoul 08826, South Korea
}

\newcommand{\bbr}{\mathbb R}

\def \bx {{\bm{x}}}
\def \bu {{\bm{u}}}
\def \bz {{\bm{z}}}
\def \bp {{\bm{p}}}
\def \bq {{\bm{q}}}

\def \d  {{\rm{d}}}
\def \e  {{\varepsilon}}

\begin{document}

\maketitle

\begin{abstract}%
 In this paper, we introduce Hamilton--Jacobi--Bellman (HJB) equations for Q-functions in continuous-time optimal control problems with Lipschitz continuous controls. 
 The standard Q-function used in reinforcement learning is shown to be the unique viscosity solution of the HJB equation. 
 A necessary and sufficient condition for optimality is provided using the viscosity solution framework. 
 By using the HJB equation, we develop a Q-learning method for continuous-time dynamical systems. 
 A DQN-like algorithm is also proposed for high-dimensional state and control spaces. 
 The performance of the proposed Q-learning algorithm is demonstrated using 1-, 10- and 20-dimensional dynamical systems.
\end{abstract}


\section{Introduction}\label{sec:1}

Q-learning is one of the most popular reinforcement learning methods that seek efficient control policies without the knowledge of an explicit system model~\cite{Watkins1992}. 
The key idea of Q-learning is to combine  dynamic programming and stochastic approximation 
in a way to estimate
the optimal state-action value function, also called the \emph{Q-function}, by using  trajectory samples. 
For discrete-time Markov decision processes, Q-learning has been extensively studied~(see~\cite{Bertsekas2019, Matni2019} and the references therein), 
while the literature on continuous-time Q-learning is  sparse. 
In discrete time, the Bellman equation for Q-functions can be defined by using dynamic programming in a straightforward manner.
However, the corresponding Bellman equation for continuous-time Q-functions has not yet been fully characterized despite some prior attempts using HJB equations~\cite{Doya2000, Munos2000}. A variant of Q-function is used in \cite{LeeParkChoi12, Mehta2009}, which has a different meaning from the Q-function in reinforcement learning. In other literature, a Q-function similar to that of reinforcement learning was introduced, but with function-valued control input \cite{Palanisamy2015} or heavily utilizing the linear-time-invariant (LTI) system structure \cite{Vamvoudakis2017}. A similar model-free approach for  LTI systems has been also studied in \cite{jiang2012computational,vrabie2009adaptive}, although an associated Q-function is not specifically defined. 
A continuous-time Q-function was also considered to prove the convergence of stochastic approximation \cite{devraj2017zap}. 

In this paper, we consider continuous-time deterministic optimal control problems with Lipschitz continuous controls. 
We show that the associated Q-function corresponds to the unique viscosity solution of a Hamilton--Jacobi--Bellman (HJB) equation in a particular form.
In the viscosity solution framework, even when it is not differentiable,  the Q-function can be used to verify the optimality of a given control and to design an optimal control strategy.  
We use the proposed HJB equation to derive an integral equation that the optimal Q-function and optimal control trajectories should satisfy. Based on this equation, we propose a Q-learning algorithm for continuous-time dynamical systems. 
For high-dimensional state and control spaces, we also propose a DQN-like algorithm by using deep neural networks (DNNs) as a function approximator~\cite{Mnih15}.
This opens a new avenue of research that connects viscosity solution theory for HJB equations and Q-learning domain.
The performance of the proposed Q-learning algorithm is tested through a set of numerical experiments with 1-, 10- and 20-dimensional  systems.

\section{Continuous-Time Q-Functions and HJB Equations}

Consider a controlled dynamical system of the form
\begin{equation}\label{dyn-con}
\dot{x}(t)=f(x(t),u(t)),\quad t>0,
\end{equation}
where $x(t)\in \bbr^n$ is the system state and $u(t)\in \bbr^m$ is the control input. Let $\mathbb{U} := \{ u: [0, T] \to \bbr^m \mid \mbox{$u$ measurable}\}$ be the set of admissible controls.
The standard finite-horizon optimal control problem can be formulated as
\begin{equation}\label{opt}
\inf_{u\in \mathbb{U}_1} J_{\bx}(u):=\inf_{u\in\mathbb{U}_1}\left\{\int_0^T r(x(t),u(t))\,\d t+q(x(T))\right\}
\end{equation}
with $x(0)=\bx$, where $r:\bbr^n\times\bbr^m\to \bbr$ and $q:\bbr^n\to\bbr$ are running and terminal cost functions of interest, respectively, and $\mathbb{U}_1$  is a subset of $\mathbb{U}$.
The Q-function $Q:\bbr^n\times\bbr^m\times[0,T]\to \bbr$ of \eqref{opt} is defined by
\begin{equation}\label{Q-function}
Q(\bx,\bu,t):=\inf_{u\in\mathbb{U}_1} \left\{\int_t^T r(x(s),u(s))\,\d s+q(x(T))~\Big|~x(t)=\bx,u(t)=\bu\right\},
\end{equation}
which represents the minimal cost incurred from time $t$ to $T$ when starting from $x(t) = \bm{x}$ with $u(t) = \bm{u}$.
In particular, when $\mathbb{U}_1 = \mathbb{U}$, the Q-function reduces to the standard optimal value function  $v:\bbr^n\times[0,T]\to\bbr$, defined by 
\begin{equation}\label{value}
v(\bx,t):=\inf_{u\in \mathbb{U}} \left\{\int_t^T r(x(s),u(s))\,\d s+q(x(T))~\Big|~x(t)=\bx\right\}.
\end{equation}

\begin{proposition}\label{P2.1}
	Suppose that $\mathbb{U}_1=\mathbb{U}$. Then, the Q-function \eqref{Q-function} corresponds to $v$ for each $\bm{u} \in \bbr^m$, i.e., $Q(\bx,\bu,t)=v(\bx,t)$ for all $(\bx, \bu, t) \in  \bbr^n \times \bbr^m \times [0,T]$.
\end{proposition}
\begin{proof}
	Fix $(\bx, \bu, t) \in  \bbr^n \times \bbr^m \times [0,T]$.
	Let $\e$ be an arbitrary positive constant. Then, there exists $u\in \mathbb{U}$ such that $\int_t^Tr(x(s),u(s))\,\d s+q(x(T))<v(\bx,t)+\e$,
	where $x(s)$ satisfies \eqref{dyn-con} with $x(t)=\bx$ in the Carath\'eodory sense: $x(s)=\bx+\int_t^s f(x(\tau),u(\tau))\,\d \tau$.
	We now construct a new control $\tilde{u}\in\mathbb{U}_1=\mathbb{U}$ as
	$\tilde{u}(s):= \bu$ if $s=t$; $\tilde{u}(s):= u(s)$ if $s > t$.
	Such a modification of controls at a single point does not affect the trajectory or the total cost. Therefore, we have
	\[
	v(\bx,t)\le Q(\bx,\bu,t) \leq \int_t^Tr({x}(s), \tilde{u}(s))\,\d s+q({x}(T)) <v(\bx,t)+\e.
	\]
	Since $\e$ was arbitrary, we conclude that $v(\bx,t)=Q(\bx,\bu,t)$ for any $\bu\in \bbr^m$. 
\end{proof}

Thus, if $\mathbb{U}_1$ is chosen to be the entire set of measurable function $\mathbb{U}$, 
the Q-function has no additional interesting property. 
Motivated by this observation, 
we restrict the control to be a  Lipschitz continuous function. Since any Lipschitz continuous function is differentiable almost everywhere,  we define the set of admissible controls $\mathbb{U}_1\subset \mathbb{U}$ as $\mathbb{U}_1:=\left\{u\in \mathbb{U} \mid \left\|\dot{u}\right\|_{L^\infty}\le M ~ \mbox{a.e.}\right\}$, where $M$ is a fixed constant. Then, for any $u\in \mathbb{U}_1$, there exists a unique measurable function $a:[0,T]\to \bbr^m$ with $|a(t)|\le M$ such that the following ODE holds a.e.:
$\dot{u}(t)=a(t)$, $0\le t\le T$.
{Thus, from now on, we will focus on the optimal control problem \eqref{opt} when the control $u$ is Lipschitz continuous such that $|\dot{u}(t)|\le M$ a.e.}
By using  the dynamic programming principle, we can deduce that
\begin{align}
\begin{aligned}\label{DPP}
Q(\bx,\bu,t)
&=\inf_{u\in \mathbb{U}_1} \left\{\int_t^{t+h} r(x(s),u(s))\,\d s+ Q(x(t+h),u(t+h),t+h) \Big|  x(t)=\bx,u(t)=\bu\right\}.
\end{aligned}
\end{align}
To derive the Hamilton-Jacobi equation that the Q-function satisfies, suppose for a moment that $Q\in C^1(\bbr^n\times\bbr^m\times[0,T])$. We will discard this regularity assumption on $Q$ by introducing the viscosity solution framework in Section~\ref{sec:vis}. 
Then, the Taylor expansion of $Q$ in \eqref{DPP} yields $
\inf_{u\in \mathbb{U}_1}\left\{\frac{1}{h}\int_t^{t+h}r(x(s),u(s))\,\d s+\partial_t Q+\nabla_\bx Q\cdot f(\bx,\bu)+\nabla_\bu Q \cdot \dot{u}(t)+O(h)\right\}=0$. Letting $h$ tend to zero, we arrive at the following HJB equation for the Q-function:
\[
\partial_t Q +\nabla_\bx Q \cdot f(\bx,\bu)+\inf_{\mathbf{a} \in \mathbb{R}^m,|\mathbf{a}|\le M}\{\nabla_\bu Q\cdot \mathbf{a}\}+r(\bx,\bu)=0.
\]
Note that $\mathbf{a}^* =-M{\nabla_\bu Q}/{|\nabla_\bu Q|}$ minimizes the Hamiltonian, and thus the HJB equation can be expressed as
\begin{equation}\label{HJB-Q}
\partial_t Q+\nabla_\bx Q\cdot f(\bx,\bu)-M|\nabla_\bu Q|+r(\bx,\bu)=0.
\end{equation}
In what follows, we uncover several mathematical properties of the HJB equation \eqref{HJB-Q} and the Q-function.

\subsection{Viscosity Solution: Existence and Uniqueness}\label{sec:vis}

In general, the Q-function is not a $C^1$-function. 
As a weak solution of the HJB equation, we use the framework of \emph{viscosity solutions}~\cite{CrandallEvansLions84, CrandallLions83}.
We begin by defining the viscosity solution of \eqref{HJB-Q} in the following standard manner~\cite{Bardi97, Evans2010}:

\begin{definition}\label{D2.1}
	A continuous function $Q:\bbr^n\times \bbr^m\times[0,T]\to \bbr$ is a \emph{viscosity solution} of \eqref{HJB-Q} if
	\begin{enumerate}
		\item $Q(\bx,\bu,T)=q(\bx)$ for all $\bu\in \bbr^m$.
		\item For any $R\in C^1(\bbr^n\times\bbr^m\times(0,T))$, if $Q-R$ has a local maximum at  $(\bx_0,\bu_0,t_0)$, then $\partial_tR(\bx_0,\bu_0,t_0)+\nabla_\bx R(\bx_0,\bu_0,t_0)\cdot f(\bx_0,\bu_0)-M|\nabla_\bu R(\bx_0,\bu_0,t_0)|+r(\bx_0,\bu_0)\ge0$.
		\item For any $R\in C^1(\bbr^n\times\bbr^m\times(0,T))$, if $Q-R$ has a local minimum at  $(\bx_0,\bu_0,t_0)$, then $\partial_t R(\bx_0,\bu_0,t_0)+\nabla_\bx R(\bx_0,\bu_0,t_0)\cdot f(\bx_0,\bu_0)-M|\nabla_\bu R(\bx_0,\bu_0,t_0)|+r(\bx_0,\bu_0)\le0$.
	\end{enumerate}
\end{definition}
From now on, we assume the following regularity conditions on $f$, $r$ and $q$:\footnote{These assumptions can be relaxed by using a modulus associated with each function as in \cite[Chapter III.1--3]{Bardi97}.}
\begin{itemize}
	\item $(A1)$ The functions $f$, $r$ and $q$ are bounded: $\|f\|_{L^\infty}+\|r\|_{L^\infty}+\|q\|_{L^\infty}<C$.
	
	\item $(A2)$ The functions $f$, $r$ and $q$ are Lipschitz continuous: $\|f\|_{\textup{Lip}}+\|r\|_{\textup{Lip}}+\|q\|_{\textup{Lip}}<C$, where $\|\cdot\|_{\textup{Lip}}$ is a Lipschitz constant of argument.
\end{itemize}
Then, the HJB equation~\eqref{HJB-Q} has a unique viscosity solution, which corresponds to the Q-function.
\begin{theorem}\label{T1}
	The  Q-function~\eqref{Q-function} is the unique viscosity solution of the HJB equation \eqref{HJB-Q}. Moreover, it is a bounded and Lipschitz continuous function.
\end{theorem}

\begin{proof}
	To accommodate the Lipschitz continuity constraint on controls, we consider an augmented system $\dot{x}(t)=f(x(t),u(t)),\dot{u}(t)=a(t)$ with $|a(t)|\le M$, where $u(t)$ and $a(t)$ are  interpreted as a new state and a new input, respectively. 
	Let $z(t) := (x(t), u(t))$ be the augmented state, and $F(\bz, {\bm a}) := (f(\bx, \bu), {\bm a})$ be the augmented vector field. 
	Then, the Q-function can be expressed as $Q(\bz,t)=\inf_{|a|\le M} \left\{\int_t^T r(z(s))\,\d s +\tilde{q}(z(T))~\Big|~z(t)=\bz\right\}$, where $\tilde{q}(\bz)=q(\bx)$.
	The HJB equation \eqref{HJB-Q} can be rewritten as $\partial_t Q+ H(\nabla_\bz Q,\bz)=0$ with $Q(\bm{z}, T) \equiv \tilde{q}(\bm{z})$, where the Hamiltonian $H=H(\bp,\bz)$ is defined by
	$H(\bp,\bz) :=H(\bp_1,\bp_2,\bz_1,\bz_2)=\bp_1\cdot f(\bz_1,\bz_2)-M|\bp_2|+r(\bz_1,\bz_2)$.
	By the assumptions (A1) and (A2), we have
	\begin{align*}
	|H(\bp,\bz)-H(\bq,\bz)|&\le|\bp_1-\bq_1||f(\bz_1,\bz_2)|+M|\bp_2-\bq_2| \le (M+\|f\|_{L^\infty})|\bp-\bq|,\\
	|H(\bp,\bz)-H(\bp,\bm{y})|& \le \left(|\bp|\|f\|_{\textup{Lip}}+\|r\|_{\textup{Lip}}\right)|\bz-\bm{y}|.
	\end{align*}
	These imply that the Hamiltonian satisfies the Lipschitz continuity conditions, and thus the standard proof for the existence and the uniqueness of viscosity solution can be directly used  (e.g.,  \cite{Evans2010}).
	Furthermore,  by \cite{Bardi97}, the Q-function corresponds to the unique viscosity solution. The boundedness and the Lipschitz continuity of the Q-function can be proved as in \cite{Evans2010}.
\end{proof}

\subsection{Asymptotic Consistency}

We now discuss the convergence of the Q-function to the optimal value function  as the Lipschitz constant $M$ tends to $\infty$. This convergence property demonstrates that the proposed HJB framework is asymptotically consistent with our observation in Proposition~\ref{P2.1}.
We parametrize the Q-function by $\e:= {1}/{M}$ so that the scaling becomes similar to that of the other classical singular limit problem.
More precisely, let
\begin{equation}\label{Q-para}
Q^\e(\bx,\bu,t) :=\inf_{u\in \mathbb{U}_1^\e}\left\{\int_t^T r(x(s),u(s))\,\d s+q(x(T))~\Big|~x(t)=\bx,~u(t)=\bu\right\},
\end{equation}
where
$\mathbb{U}_1^\e:= \{u\in \mathbb{U} \mid  \|\dot{u} \|_{L^\infty}\le 1/{\e},~ \mbox{a.e.} \}.$
 Since $\mathbb{U}_1^{\e_1}\subseteq \mathbb{U}_1^{\e_2}$ for any $\e_1 \geq \e_2 > 0$, it is straightforward to observe that $Q^{\e_1}(\bx,\bu,t)\ge Q^{\e_2}(\bx,\bu,t)$. 
We also notice that $Q^\epsilon$ is a bounded function under Assumption (A1).
Therefore, by the monotone convergence theorem,
there exists a limit function $Q^0(\bx,\bu, t)$ such that
\[Q^0(\bx,\bu,t) =\lim_{\e\to0}Q^\e(\bx,\bu,t) \quad \forall (\bx, \bu, t) \in \mathbb{R}^n \times \bbr^m \times [0,T].
\]
The limit function corresponds to the optimal value function~\eqref{value} without the Lipschitz continuity constraint on controls.

\begin{theorem}\label{T3}
	For any $(\bx, \bu, t) \in \bbr^n \times \bbr^m \times [0,T]$, we have $Q^0(\bx,\bu,t)=v(\bx,t)$.
\end{theorem}
\begin{proof}
Since the argument in the proof of \cite[Theorem 4.1 in Ch. 7]{Bardi97} can be used, 
we omit the proof.
\end{proof}

\subsection{Optimal Controls}

To characterize a necessary and sufficient condition for optimality of a control $u \in \mathbb{U}_1$, we consider the function $g_{u}(s;\bx,\bu,t):=\int_t^s r(x(\tau),u  (\tau))\,\d \tau +Q(x(s),u  (s),s)$, with $ x(t)=\bx,u(t)=\bu$. By~\eqref{DPP}, we deduce that the control $u$ is optimal if and only if
$s \mapsto g_u (s; \bx, \bu, t)$ is a constant function for each $(\bx, \bu, t)$. 
On the other hand, the dynamic programming principle implies that the function $s \mapsto g_u(s;\bx,\bu,t)$ is non-decreasing for any control $u\in \mathbb{U}_1$.
Thus, the control $u$ is optimal if and only if the function
$s\mapsto g_u(s;\bx,\bu,t)$ is non-increasing.
If the Q-function is differentiable, this implies $\frac{\d}{\d s}g_u(s;\bx,\bu,t)=r(x(s),u(s))+\nabla_\bx Q\cdot f(x(s),u(s))+\nabla_\bu Q \cdot \dot{u} +\partial_t Q \le 0$.
Since Q-function satisfies HJB equation \eqref{HJB-Q}, we have
\begin{align*}
0&= \partial_tQ+\nabla_\bx Q\cdot f(x(s),u(s))-M|\nabla_\bu Q|+r(x(s),u(s))\\
&\le \partial_t Q+\nabla_\bx Q\cdot f(x(s),u(s))+\nabla_\bu Q\cdot\dot{u}+r(x(s),u(s))\le0.
\end{align*}
Therefore, when $Q$ is differentiable,
$u \in \mathbb{U}_1$ is optimal if and only if $\dot{u}=-M {\nabla_\bu Q}/{|\nabla_\bu Q|}$ with $u(t)=\bu$.
To obtain the rigorous principle of optimality when $Q$ is not differentiable, we use generalized derivatives of the Q-function, namely, sub- and superdifferntials of $Q$.
The following optimality theorem is a direct application of the classical results in \cite[Theorem 3.39, Ch. 3]{Bardi97}.

\begin{theorem}\label{T2}
	Suppose that $f$ and $q$ are continuously differentiable. Then, the trajectory-control pair $(x^*,u^*)$ is optimal if and only if
	\begin{equation}\label{optimal_u}
	\dot{u}^*(s)=-M\frac{p_1}{|p_1|} \quad \forall p=(p_0,p_1,p_2)\in D^{\pm}Q(x^*(s),u^*(s),s),\quad \mbox{a.e.}~s\ge t.
	\end{equation}
\end{theorem}

\begin{proof}
	Note that the controlled system is equivalent to the extended dynamics $\dot{z}(t)=F(z(t),a(t))$ defined in the proof of Theorem \ref{T1}. Then, the assertions can be obtained by directly applying \cite[Theorem 3.39 in Chapter 3]{Bardi97} to this augmented system. 
\end{proof}

	At a point $(\bx,\bu,t)$ where $Q$ is differentiable, the sub- and superdifferentials of $Q$ are identical to the classical derivative of $Q$. Thus, at such a point, we can construct the optimal control by using $\dot{u}^*=-M \frac{\nabla_\bu Q}{|\nabla_\bu Q|}$. At a point where $Q$ is not differentiable, one can choose any control given by \eqref{optimal_u}.

\section{Q-Learning Using the HJB Equation}\label{sec:4}
In the infinite-horizon case,  we consider the following discounted cost function (with $\gamma > 0$):
\[
J_\bx(u)=\int_0^\infty e^{-\gamma t}r(x(t),u(t))\,\d t,\quad x(0)=\bx,
\]
and the Q-function is defined by 
\[
 Q(\bx,\bu):=\inf_{u\in \mathbb{U}_1} \left  \{\int_0^\infty e^{-\gamma t} r(x(t),u(t))\d t 
\: | \: x(0)=\bx,u(0)=\bu \right \}.
\]
 Again, using the dynamic programming principle, we can derive the following HJB equation:
\begin{equation}\label{HJB-Q-inf}
\gamma Q(\bx,\bu)-r(\bx,\bu)-\nabla_\bx Q\cdot f(\bx,\bu)+M|\nabla_\bu Q|=0.
\end{equation}
As in Theorem~\ref{T1}, we can show that the Q-function is the unique viscosity solution of the HJB equation~\eqref{HJB-Q-inf}.
A necessary and sufficient condition for optimality can be characterized in a way similar to Theorem~\ref{T2}.

We now discuss how the HJB equation~\eqref{HJB-Q-inf} can be used to design a Q-learning algorithm for estimating $Q(\bx,\bu)$ using sample trajectories. To provide the essential idea, we assume for a moment that $Q$ is differentiable. We then have
\begin{align*}
\frac{\d}{\d t}Q(x(t),u(t))&=\nabla_\bx Q\cdot f(x(t),u(t))+\nabla_\bu Q\cdot \dot{u}(t)\\
&=\gamma Q(x(t),u(t))-r(x(t),u(t))+M|\nabla_\bu Q|+\nabla_\bu Q\cdot \dot{u}(t).
\end{align*}
Suppose now that an optimal control is employed, i.e.,
$\dot{u} :=-M {\nabla_\bu Q}/{|\nabla_\bu Q|}$. 
Then, the time derivative of Q-function along the trajectory is further simplified as
\begin{equation}\label{ode}
\frac{\d}{\d t}Q(x(t),u(t))=\gamma Q(x(t),u(t))-r(x(t),u(t)).
\end{equation}
By integrating \eqref{ode} along the optimal trajectory-control pair, we obtain
\begin{equation}\label{bellman}
Q(\bx,\bu)=\int_0^t e^{-\gamma t} r(x(t),u(t))\,\d t +e^{-\gamma t} Q(x(t),u(t)),\quad \forall t\ge0.
\end{equation}
However, since the optimal Q-function and optimal trajectories are unknown {\it a priori}, we iteratively update the Q-function and the control $u$ with $\dot{u} =-M {\nabla_\bu Q}/{|\nabla_\bu Q|}$ by using sample data. 
The iteration is based on the equation \eqref{bellman}, which is the characterizing equation of the optimal Q-function. Specifically, for a given control $\dot{u}=a$, we obtain the system trajectory starting from randomly chosen $(\bx,\bu)$ for time interval $[0,h]$ with small $h>0$ and collect the sample data $\int_0^he^{-\gamma t}r(x(t),u(t))\,\d t$, $\bx':=x(h)$ and $\bu':=u(h)$. We then update a new estimate for Q-function by using \eqref{bellman} and sample data as $Q(\bx,\bu)\leftarrow \int_0^h e^{-\gamma t}r(x(t),u(t))\,\d t +e^{-\gamma h}Q(\bx',\bu')$ so that \eqref{bellman}  holds asymptotically. 

To handle high-dimensional state and control spaces, we propose a DQN-like algorithm by using
DNNs as a  function approximator.
Let $Q_\theta (\bx,\bu)$ denote the approximate Q-function 
parameterized by $\theta$. 
We update the network parameter $\theta$ by minimizing the mean squared error (MSE) loss between $Q_\theta(\bx,\bu)$ and the target $y=\int_0^h e^{-\gamma t}r(x(t),u(t))\,\d t +e^{-\gamma h}Q_\theta(\bx',\bu')$. To enhance the stability of learning procedure, we use $K$ sample points of $(\bx_i,\bu_i)$ for defining MSE loss $L(\theta):=\frac{1}{K}\sum_{i=1}^K (Q_\theta(\bx_i,\bu_i)-y_i)^2$
and introduce the target network parameter  $\theta^-$ in estimating target value $y_i$ as in DQN~\cite{Mnih15}. The target network is slowly updated as a weighted sum of $\theta$ and itself as in \cite{Lillicrap16}. The algorithm minimizes the error between the left and right-hand sides of \eqref{bellman} for each iteration, making the $Q_\theta$ asymptotically satisfies \eqref{bellman} as much as possible.  
The overall procedures are summarized in Algorithm~\ref{alg1}. 
Note that this algorithm does not need the knowledge of an explicit system model as in discrete-time Q-learning or DQN.  

\begin{algorithm2e}
	\caption{Continuous-Time Q-Learning}
	\label{alg1}
	Randomly initialize network parameter $\theta$ and initialize the target network parameter $\theta^-\leftarrow \theta$;\\
	Define the domain $\Omega:=[\bx_{\min},\bx_{\max}]^n\times [\bu_{\min},\bu_{\max}]^m$;\\
	\For{iter = 1 to $N$}{
		\uIf{$\nabla_\bu Q_\theta(\bx,\bu) \neq 0$}{
 Set  $a_\theta(\bx,\bu)$ as $-M\frac{\nabla_\bu Q_\theta(\bx,\bu)}{|\nabla_\bu Q_\theta(\bx,\bu)|}$;
		}
		\Else {Set  $a_\theta(\bx,\bu)$ as a random vector of length $M$;
}
Randomly choose $K$ samples of $(\bx_i,\bu_i)\in \Omega$;\\ 
		\For{i=1 to $K$}{
			Obtain discounted running cost $R_i:=\int_0^h e^{-\gamma t}r(x_i(t),u_i(t))\,\d t$ and the terminal point $(x_i(h),u_i(h))$ by using  $a_\theta$;\\
			Set the target as $y_i:=R_i+e^{-\gamma h}Q_{\theta^-}(x_i(h),u_i(h))$;
		}
		Update the network parameter $\theta$ by minimizing the MSE loss: $\frac{1}{K}\sum_{i}(Q_\theta(\bx_i,\bu_i)-y_i)^2$;\\
		Update the target network parameter as $\theta^-\leftarrow \tau \theta+(1-\tau)\theta^-$;
	}
\end{algorithm2e}

\section{Numerical Experiments}

We consider the following linear system with an exponentially discounted quadratic cost:
 \[\dot{x}(t)=Ax(t)+Bu(t),\quad J_\bx(u):=\int_0^\infty e^{-\gamma t} (|x(t)|^2+|u(t)|^2)\,\d t,\quad x(0)=\bx,\quad u(0)=\bu,\]
where $x(t)\in \bbr^n, u(t)\in\bbr^m, A\in \bbr^{n\times n}$ and $B\in \bbr^{n\times m}$. We restrict the control $u$ as a Lipschitz continuous function with Lipschitz constant 1. The parameters for simulation are chosen as $\bx_{\min}=\bu_{\min}=-1, \bx_{\max}=\bu_{\max}=1, \gamma=0.1, h=0.05, \tau=10^{-2}, K=10$ and $N=10^3$.
As the DNNs for approximating Q-functions, we use fully connected networks consisting of an input layer with $m+n$ nodes, two hidden layers with 128 nodes and an output layer with a single node. For the two hidden layers, we use ReLU activation function. For training the networks, we use Adam optimizer with a learning rate $10^{-3}$~\cite{Kingma2015}.

\subsection{One-dimensional problem}
As a toy example for sanity check, we first consider a one dimensional model, where $n=m=1$, $A=0$ and $B=1$. 
In order to measure the performance of control, we fix the initial state and control $(\bx,\bu)=(1,1)$ and we integrate the running cost over $[0, 10]$. Figure \ref{fig:1}\hyperref[fig:1] {(a)} shows the log of  costs  at each iteration of Algorithm \ref{alg1}. The solid line represents the learning curve averaged over five different trials and the shaded region represents the minimum and maximum of the cost over different trials. The running cost rapidly decreases as the network parameters $\theta$ is learned. {We also note that the variation over five different trials vanishes as the parameters are learned.}
Figure \ref{fig:1}\hyperref[fig:1] {(b)} shows the trajectory of $x(t)$, $u(t)$, and $a(t)$ generated by $\dot{u}(t)=a(t):=-\frac{\nabla_{\bm{u}} Q_{\theta}(x(t),u(t))}{|\nabla_{\bm{u}} Q_{\theta}(x(t),u(t))|}$ using the learned Q-function $Q_\theta(\bx,\bu)$ after $10^3$ iteration. 
The optimal control for this one dimensional problem is to drive both $x(t)$ and $u(t)$ as soon as possible to 0. Starting from $(\bx,\bu)=(1,1)$, this can be done by first driving the control $u(t)$ to negative value so that  $x(t)$ moves towards the origin, and then reducing the absolute value of $u(t)$ so that both $x(t)$ and $u(t)$ approaches 0 asymptotically. Such a behavior is observed in Figure \ref{fig:1}\hyperref[fig:1]{\textcolor{blue}{(b)}}. This confirms that the learned policy is near optimal.

\begin{figure}[tb]
	\centering
	\mbox{
		\subfigure[]{\includegraphics[height=1.75in]{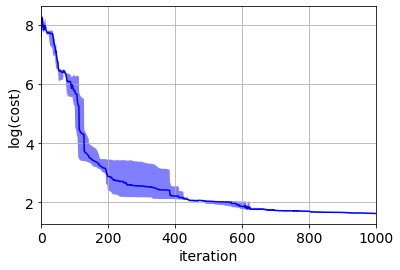}
			\label{fig:1a}}
		\subfigure[]{\includegraphics[height=1.75in]{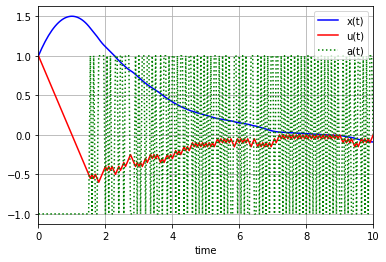}
			\label{fig:1b}}
	}
	\centering
	\caption{Numerical experiment result for $n=m=1$: (a) learning curve; (b) controlled trajectory of $x(t), u(t)$ and $a(t)$.}
	\label{fig:1}
\end{figure}

\subsection{10- and 20-dimensional problems}
For 10- and 20-dimensional problems, {where $n=m=10$ or $n=m=20$, respectively,} we set the coefficient matrices $A\in \bbr^{n\times n}$ and $B\in \bbr^{n\times m}$ such that each element of $A$ is randomly sampled from $U[0,1]$ and then multiplied by 0.1, and each element of $B$ is randomly sampled from the uniform distribution $U[0,1]$ and then multiplied by 5:
\[A_{ij}=0.1X_{ij},\quad B_{ij}=5Y_{ij},\quad X_{ij},Y_{ij}\sim U[0,1].\]


\begin{figure}[ht]
		\vspace{-0.1in}
	\centering
	\mbox{
		\subfigure[]{\includegraphics[height=1.75in]{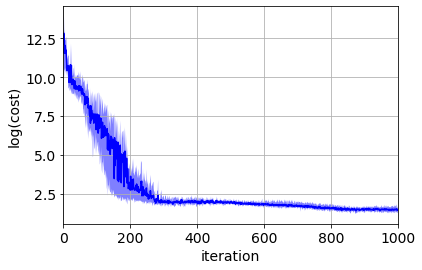}
			\label{fig:3a}}
		\subfigure[]{\includegraphics[height=1.75in]{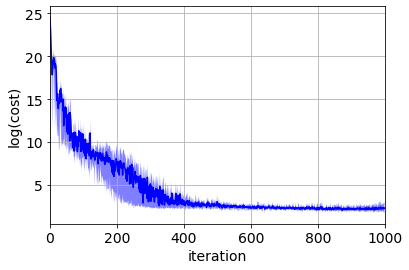}
			\label{fig:3b}}
	}
	\centering
		\vspace{-0.1in}
		\caption{Numerical experiment result: (a) learning curve for $n=m=10$; (b) learning curve for $n=m=20$.}
	\label{fig:3}
\end{figure}

We integrate the running cost from randomly chosen initial position and control $(\bx,\bu)\in[0,0.1]^{m+n}$.  
{The learning curves are shown in Figure \ref{fig:3}. As in the one-dimensional problem, the solid curves denote the average of the running costs over five trials and the shaded regions represent the minimum and maximum of the running costs.}
The results show that the cost decreases super-exponentially in both 10- and 20-dimensional problems. 
Note that the $y$-axis of Figure \ref{fig:3}\hyperref[fig:3]{\textcolor{blue}{(a)}} and Figure \ref{fig:3}\hyperref[fig:3]{\textcolor{blue}{(b)}} are plotted with log-scale, and thus the decay of cost is rapid. 
{Again, the results confirm that our continuous-time Q-learning algorithm presents the desired performance.}

%
%
%
%
%

\section{Conclusion}
We introduced a Q-function for continuous-time deterministic optimal control problems with Lipschitz continuous control.
By using the dynamic programming principle, we derived 
the corresponding HJB equation and showed that its unique viscosity solution corresponds to the Q-function.
An optimality condition was also characterized in terms of the Q-function without the knowledge of system models. 
Using the HJB equation and the optimality condition, we construct 
a Q-learning algorithm and its DQN-like approximate version. 
 The simulation results show that the proposed Q-learning algorithm is fast and stable, and that
 the learned  controller presents a good performance, even in the 20-dimensional case. 

\acks{This work was supported in part by the Creative-Pioneering Researchers Program through SNU, and the Basic Research Lab Program through the National Research Foundation of Korea funded by the MSIT(2020R1C1C1009766).}

\bibliography{reference}

\begin{thebibliography}{19}
\providecommand{\natexlab}[1]{#1}
\providecommand{\url}[1]{\texttt{#1}}
\expandafter\ifx\csname urlstyle\endcsname\relax
  \providecommand{\doi}[1]{doi: #1}\else
  \providecommand{\doi}{doi: \begingroup \urlstyle{rm}\Url}\fi

\bibitem[Bardi and Capuzzo-Dolcetta(1997)]{Bardi97}
M.~Bardi and I.~Capuzzo-Dolcetta.
\newblock \emph{Optimal Control and Viscosity Solutions of
  Hamilton-Jacobi-Bellman Equations}.
\newblock Birkh\"{a}user, 1997.

\bibitem[Bertsekas(2019)]{Bertsekas2019}
D.~P. Bertsekas.
\newblock \emph{Reinforcement learning and optimal control}.
\newblock Athena Scientific, 2019.

\bibitem[Crandall and Lions(1983)]{CrandallLions83}
M.~Crandall and P.-L. Lions.
\newblock Viscosity solutions of hamilton--jacobi equations.
\newblock \emph{Transactions of the American Mathematical Society},
  277:\penalty0 1--42, 1983.

\bibitem[Crandall et~al.(1984)Crandall, Evans, and Lions]{CrandallEvansLions84}
M.~Crandall, L.~C. Evans, and P.-L. Lions.
\newblock Some properties of viscosity solutions of {Hamilton--Jacobi}
  equations.
\newblock \emph{Transactions of the American Mathematical Society},
  282:\penalty0 487--502, 1984.

\bibitem[Devraj and Meyn(2017)]{devraj2017zap}
A.~M. Devraj and S.~Meyn.
\newblock Zap q-learning.
\newblock In \emph{Advances in Neural Information Processing Systems}, pages
  2235--2244, 2017.

\bibitem[Doya(2000)]{Doya2000}
K.~Doya.
\newblock Reinforcement learning in continuous time and space.
\newblock \emph{Neural Computation}, 12\penalty0 (1):\penalty0 219--245, 2000.

\bibitem[Evans(2010)]{Evans2010}
L.~C. Evans.
\newblock \emph{Partial Differential Equations}.
\newblock American Mathematical Society, 2010.

\bibitem[Jiang and Jiang(2012)]{jiang2012computational}
Y.~Jiang and Z.-P. Jiang.
\newblock Computational adaptive optimal control for continuous-time linear
  systems with completely unknown dynamics.
\newblock \emph{Automatica}, 48\penalty0 (10):\penalty0 2699--2704, 2012.

\bibitem[Kingma and Ba(2015)]{Kingma2015}
D.~P. Kingma and J.~Ba.
\newblock Adam: A method for stochastic optimization.
\newblock In \emph{International Conferences on Learning Representation
  (ICLR)}, 2015.

\bibitem[Lee et~al.(2012)Lee, Park, and Choi]{LeeParkChoi12}
J.~Y. Lee, J.~B. Park, and Y.~H. Choi.
\newblock Integral {Q}-learning and explorized policy iteration for adaptive
  optimal control of continuous-time linear systems.
\newblock \emph{Automatica}, 48:\penalty0 2850--2859, 2012.

\bibitem[Lillicrap et~al.(2016)Lillicrap, Hunt, Pritzel, Heess, Erez, Tassa,
  Silver, and Wierstra]{Lillicrap16}
T.~P. Lillicrap, J.~J. Hunt, A.~Pritzel, N.~Heess, T.~Erez, Y.~Tassa,
  D.~Silver, and D.~Wierstra.
\newblock Continuous control with deep reinforcement learning.
\newblock In \emph{International Conferences on Learning Representation
  (ICLR)}, 2016.

\bibitem[Matni et~al.(2019)Matni, Proutiere, Rantzer, and Tu]{Matni2019}
N.~Matni, A.~Proutiere, A.~Rantzer, and S.~Tu.
\newblock From self-tuning regulators to reinforcement learning and back again.
\newblock In \emph{IEEE Conference on Decision and Control (CDC)}, 2019.

\bibitem[Mehta and Meyn(2009)]{Mehta2009}
P.~Mehta and S.~Meyn.
\newblock Q-learning and {Pontryagin's} minimum principle.
\newblock In \emph{Proceedings of the 48th IEEE Conference on Decision and
  Control}, pages 3598--3605, 2009.

\bibitem[Mnih et~al.(2015)Mnih, Kavukcuoglu, Silver, Rusu, Veness, Bellemare,
  Graves, Riedmiller, Fidjeland, Ostrovski, et~al.]{Mnih15}
V.~Mnih, K.~Kavukcuoglu, D.~Silver, A.~A. Rusu, J.~Veness, M.~G. Bellemare,
  A.~Graves, M.~Riedmiller, A.~K. Fidjeland, G.~Ostrovski, et~al.
\newblock Human-level control through deep reinforcement learning.
\newblock \emph{Nature}, 518:\penalty0 529--533, 2015.

\bibitem[Munos(2000)]{Munos2000}
R.~Munos.
\newblock A study of reinforcement learning in the continuous case by the means
  of viscosity solutions.
\newblock \emph{Machine Learning}, 40:\penalty0 265--299, 2000.

\bibitem[Palanisamy et~al.(2015)Palanisamy, Modares, Lewis, and
  Aurangzeb]{Palanisamy2015}
M.~Palanisamy, H.~Modares, F.~L. Lewis, and M.~Aurangzeb.
\newblock Continuous-time {Q}-learning for infinite-horizon discounted cost
  linear quadratic regulator problems.
\newblock \emph{IEEE Transactions on Cybernetics}, 45:\penalty0 165--176, 2015.

\bibitem[Vamvoudakis(2017)]{Vamvoudakis2017}
K.~G. Vamvoudakis.
\newblock Q-learning for continuous-time linear systems: A model-free infinite
  horizon optimal control approach.
\newblock \emph{Systems \& Control Letters}, 100:\penalty0 14--20, 2017.

\bibitem[Vrabie et~al.(2009)Vrabie, Pastravanu, Abu-Khalaf, and
  Lewis]{vrabie2009adaptive}
D.~Vrabie, O.~Pastravanu, M.~Abu-Khalaf, and F.~L. Lewis.
\newblock Adaptive optimal control for continuous-time linear systems based on
  policy iteration.
\newblock \emph{Automatica}, 45\penalty0 (2):\penalty0 477--484, 2009.

\bibitem[Watkins and Dayan(1992)]{Watkins1992}
C.~J. Watkins and P.~Dayan.
\newblock Q-learning.
\newblock \emph{Machine Learning}, 8\penalty0 (3-4):\penalty0 279--292, 1992.

\end{thebibliography}

\end{document}